\documentclass[a4paper,reqno]{amsart}

\usepackage[T1]{fontenc}
\usepackage[utf8]{inputenc}
\usepackage{amssymb,amsthm,amsmath}
\usepackage[british]{babel}
\usepackage[all]{xy}
\usepackage{calrsfs}
\usepackage{graphicx}
\usepackage{hyperref}
\usepackage{mathbbol}
\usepackage{mathabx}

\theoremstyle{plain}
\newtheorem{lemma}[subsection]{Lemma}
\newtheorem{proposition}[subsection]{Proposition}
\newtheorem{theorem}[subsection]{Theorem}
\newtheorem{corollary}[subsection]{Corollary}

\theoremstyle{definition}
\newtheorem{definition}[subsection]{Definition}

\theoremstyle{remark}
\newtheorem{remark}[subsection]{Remark}

\newcommand{\defn}{\textbf}

\newcommand{\normal}{\lhd}

\def\al{\alpha}
\def\be{\beta}

\def\de{\delta}

\def\pa{\partial}

\def\Gm{\Gamma}

\def\lra{\longrightarrow}
\def\rra{\rightarrow}

\def\otm{\otimes}

\def\uns{\underset}

\DeclareMathOperator{\Coker}{Coker}

\DeclareMathOperator{\Image}{Im}
\renewcommand{\Im}{\Image}
\DeclareMathOperator{\Ker}{Ker}

\DeclareMathOperator{\Moore}{N}
\DeclareMathOperator{\Nerve}{Ner}

\DeclareMathOperator{\Tot}{Tot}

\newcommand{\A}{A}
\newcommand{\F}{F}
\newcommand{\T}{T}
\newcommand{\U}{U}

\newcommand{\Z}{\ensuremath{\mathbb{Z}}}

\newcommand{\HH}{\ensuremath{\mathcal{H}}}
\newcommand{\LL}{\ensuremath{\mathcal{L}}}
\newcommand{\PP}{\ensuremath{\mathcal{P}}}

\newcommand{\FF}{\ensuremath{P}}

\newcommand{\Ab}{\ensuremath{\mathsf{Ab}}}
\newcommand{\Cat}{\ensuremath{\mathsf{Cat}}}
\newcommand{\Gp}{\ensuremath{\mathsf{Gp}}}
\newcommand{\Set}{\ensuremath{\mathsf{Set}}}
\newcommand{\XMod}{\ensuremath{\mathsf{XMod}}}

\newdir{>>}{{}*!/3.5pt/:(1,-.2)@^{>}*!/3.5pt/:(1,+.2)@_{>}*!/7pt/:(1,-.2)@^{>}*!/7pt/:(1,+.2)@_{>}}
\newdir{ >>}{{}*!/8pt/@{|}*!/3.5pt/:(1,-.2)@^{>}*!/3.5pt/:(1,+.2)@_{>}}
\newdir{ |>}{{}*!/-3.5pt/@{|}*!/-8pt/:(1,-.2)@^{>}*!/-8pt/:(1,+.2)@_{>}}
\newdir{ >}{{}*!/-8pt/@{>}}
\newdir{>}{{}*:(1,-.2)@^{>}*:(1,+.2)@_{>}}
\newdir{<}{{}*:(1,+.2)@^{<}*:(1,-.2)@_{<}}

\begin{document}
\date{\today}

\title[A comonadic interpretation of Baues--Ellis homology]{A comonadic interpretation of\\
Baues--Ellis homology of crossed modules}

\thanks{This work was partially supported by Agencia Estatal de Investigación (Spain), grant MTM2016-79661-P (EU supported included)}

\author{Guram Donadze}
\address[Guram Donadze]{Indian Institute of Science Education and Research (IISER-TRV), CET Campus, TVPM-16, Thiruvananthapuram, Kerala, India}
\thanks{The first author wishes to thank Prof.~Manuel Ladra Gonz\'alez and the University of Santiago de Compostela for their kind hospitality} 
\email{gdonad@gmail.com}

\author{Tim Van~der Linden}
\address[Tim Van~der Linden]{Institut de Recherche en Math\'ematique et Physique, Universit\'e catholique de Louvain, chemin du cyclotron~2 bte~L7.01.02, 1348 Louvain-la-Neuve, Belgium}
\thanks{The second author is a Research Associate of the Fonds de la Recherche Scientifique--FNRS. He wishes to thank Dr~Viji J.~Thomas with the Indian Institute of Science Education and Research at Thiruvananthapuram, and Prof.~Manuel Ladra Gonz\'alez with the University of Santiago de Compostela for their kind hospitality} 
\email{tim.vanderlinden@uclouvain.be}

\begin{abstract}
We introduce and study a homology theory of crossed modules with coefficients in an abelian crossed module. We discuss the basic properties of these new homology groups and give some applications. We then restrict our attention to the case of integral coefficients. In this case we regain the homology of crossed modules originally defined by Baues and further developed by Ellis. We show that it is an instance of Barr--Beck comonadic homology, so that we may use a result of Everaert and Gran to obtain Hopf formulae in all dimensions.
\end{abstract}

\subjclass[2010]{18G10; 18G40; 18C15; 20J06; 55N35}

\keywords{Comonadic homology; crossed module; spectral sequence; Hopf formula}

\maketitle

\section{Introduction}

The concept of a crossed module originates in the work of Whitehead~\cite{Wh}. It was introduced as an algebraic model for path-connected CW spaces whose homotopy groups are trivial in dimensions strictly above  $2$. The role of crossed modules and their importance in algebraic topology and homological algebra are nicely demonstrated in the articles~\cite{Breen, BH, Brown-Loday, LR, MW, Wh}, just to mention a few. The study of crossed modules as algebraic objects in their own right has been the subject of many investigations, amongst which the Ph.D.\ thesis of Norrie~\cite{Norrie-PhD} is prominent. She showed that some group-theoretical concepts and results have suitable
counterparts in the category of crossed modules. 

In~\cite{Baues1991}, Baues defined a (co)homology theory of crossed modules via classifying spaces, which was studied using algebraic methods by Ellis in~\cite{Ellis1992}. A different approach to (co)homo\-logy of crossed modules is due to Carrasco, Cegarra and R.-Grandje\'an. In particular, in the paper~\cite{Carrasco-Homology}, the authors develop Barr--Beck comonadic homology~\cite{Barr-Beck} of crossed modules, relative to the canonical comonad which arises from the forgetful adjunction to the category of sets, and with coefficients in the abelianisation functor. In the article~\cite{GLP} certain relationships between these two homology theories are explored.

In the present article we introduce and study a new homology theory of crossed modules---one where coefficients are taken in an abelian crossed module. In the case of integral coefficients, it restricts to the theory of Baues and Ellis. We furthermore show that, in this particular case, homology forms a non-abelian derived functor, which may be obtained via comonadic homology as in the approach of Carrasco, Cegarra and R.-Grandje\'an. Since the functor which is being derived turns out to be ``sufficiently nice'', we may apply results of Everaert and Gran~\cite{Everaert-Gran-TT} leading to an interpretation of the homology objects via higher Hopf formulae, extending those of Brown and Ellis~\cite{Brown-Ellis, Donadze-Inassaridze-Porter, EGVdL, TomasThesis}.

Let us give a quick overview of the results in the paper. Consider a crossed module $(H, G, \mu\colon{H\to G})$ acting on an abelian crossed module $(C, A, \nu)$; this means that $(C, A, \nu)$ carries the structure of a Beck module~\cite{Beck} over $(H, G, \mu)$. Then we denote by $H_n((H, G, \mu), (C, A, \nu))$ the $n$th homology group of $(H, G, \mu)$ with coefficients in $(C, A, \nu)$, defined as the homology of the total complex of some suitable double complex---see Section~\ref{Homology} for details. If $(C, A, \nu)=(0, \Z, 0)$ and $(H, G, \mu)$ acts trivially on $(0, \Z, 0)$, we regain the homology of $(H, G, \mu)$ with integral coefficients as defined by Baues and Ellis.

For an inclusion crossed module $(N, G, \sigma)$ acting on an abelian crossed module $(C, A, \nu)$, we obtain the exact sequence of abelian groups
\begin{align*}
&\cdots \to H_{n-1}(G/N, \Ker \nu) \to H_n((N, G, \sigma), (C, A, \nu))\to H_n(G/N, A/\Im\nu)\\
&\to H_{n-2}(G/N, \Ker \nu) \to \cdots
\end{align*}
where $H_{n}(G/N, \Ker \nu)$ and $H_n(G/N, A/\Im\nu)$ denote the Eilenberg--Mac\,Lane homologies of $G/N$ with coefficients in
$\Ker \nu$ and in $A/\Im\nu$, respectively: see Proposition~\ref{prop3}.
 
A first main result of this paper is the Lyndon--Hochschild--Serre spectral sequence adapted to crossed modules (Theorem~\ref{theo1}). We use it to show that if two
crossed modules are (weakly) homotopically equivalent, then their homology groups are isomorphic (Proposition~\ref{prop6}). Another consequence of this spectral sequence is a five-term exact sequence relating the homology groups of crossed modules in low degrees (Proposition~\ref{prop7}). We also prove that the third homology of a crossed module (as defined by Baues and Ellis) is isomorphic to the first homology of the same crossed module with certain coefficients (Proposition~\ref{prop8}). 

In Section~\ref{Section Left Derived Functors} we show that the homology groups (with integral coefficients) of a crossed module can be described via non-abelian left derived functors (Theorem~\ref{theo2}). The comonad on the category of crossed modules is still the one induced by the adjunction to sets as in~\cite{Carrasco-Homology}. However, the functor which is being derived is different: it sends a crossed module $(H, G, \mu)$ to the group $G/(\mu(H)\cdot[G, G])$, where $\mu(H)\normal G$ is the image of $\mu$ and $[G, G]\normal G$ is the ordinary commutator subgroup. This functor $\A\colon{\XMod\to \Ab}$ turns out to be an example considered in~\cite{Everaert-Gran-TT}. The theory developed there leads to an interpretation of Baues--Ellis homology via higher Hopf formulae (Theorem~\ref{Theorem Hopf}).

\section{Crossed modules, their nerves and their actions} \label{sec1}

A \defn{crossed module} $(H, G,\mu)$ is a homomorphism of groups $\mu \colon H\rra G$, together with an action of $G$ on
$H$, such that the identities 
\[
\begin{cases}
\mu (^gh)= g \mu(h) g^{-1} & \text{(Precrossed module identity)}\\
^{\mu (h)}h'=h h' h^{-1} & \text{(Peiffer identity)}
\end{cases}
\]
hold for all $g\in G$ and $h$, $h'\in H$. A common instance of a crossed module is that of a group~$G$ possessing a normal subgroup $N\normal G$; the inclusion homomorphism $N\hookrightarrow G$ is a crossed module with
$G$ acting on $N$ by the conjugation in $G$, called the \defn{inclusion crossed module} of $N\normal G$. More generally, given any crossed module $(H,G,\mu)$, we may consider the semidirect product $H\rtimes G$. Then there are
homomorphisms $s\colon H\rtimes G\to G\colon (h, g)\mapsto g$, $t\colon H\rtimes G\to G\colon (h, g)\mapsto \mu(h)g$ and $e\colon G\to H\rtimes G\colon g\to (1,g)$ and a binary operation $(h', g')\circ (h, g) = (h h', g)$ for all pairs $(h, g)$, $(h', g')\in H\rtimes G$ such that $\mu(h) g = g'$.
This composition $\circ$ with the source map $s$, target map $t$ and identity map $e$ constitutes an internal category in the category of groups. Conversely, given any internal category $(M,G,s,t,e)$ in~$\Gp$, the induced map ${tk\colon \Ker s\to G}$ where $k\colon \Ker s\to M$ denotes the kernel of~$s$, together with the restriction to $G$ of the conjugation action of $M$ on $\Ker s$, forms a crossed module.

A \defn{morphism of crossed modules} $(\rho,\nu)\colon (H,G,\mu)\to(H',G',\mu')$ is a commutative square of groups
\[
\xymatrix{
H \ar[r]^-{\mu}\ar[d]_-\rho & G \ar[d]^-{\nu}\\
H' \ar[r]_-{\mu'}& G'}
\]
such that $\rho(^g h)=\;^{\nu(g)}\rho(h)$ for all $g\in G$, $h\in H$. The category of crossed modules with their morphisms is denoted $\XMod$. It is well known to be a variety of $\Omega$-groups in the sense of Higgins~\cite{Higgins, Carrasco-Homology} and as such it is an example of a semi-abelian variety of algebras~\cite{Janelidze-Marki-Tholen}. Furthermore~\cite{LR,MacLane}, it is equivalent to the category $\Cat(\Gp)$ of internal categories in the category of groups, essentially via the explanation recalled above.

Given a crossed module $(H,G,\mu)$, the nerve of its category structure forms the simplicial group
$\Nerve_*(H,G,\mu)$, called the \defn{nerve} of the crossed module $(H,G,\mu)$. It is not difficult to see that $\Nerve_n(H,G,\mu)=H\rtimes (\cdots (H\rtimes G)\cdots)$ with $n$ semidirect factors of $H$, and that the face and degeneracy homomorphisms are given by
\begin{align*}
d_0(h_1, \ldots, h_n, g) &= (h_2, \ldots, h_n, g),\\
d_i(h_1, \ldots, h_n, g) &= (h_1, \ldots, h_i h_{i+1}, \ldots, h_n, g), &0<i<n,\\
d_n(h_1, \ldots, h_n, g) &= (h_1, \ldots, h_{n-1}, \mu(h_n) g),\\
s_i(h_1, \ldots, h_n, g) &= (h_1, \ldots, h_i, 1, h_{i+1}, \ldots,
h_n, g), &0\le i\le n.
\end{align*}

Given a simplicial group $G_*=(G_n, d_n^i, s_n^i)$, recall that its \defn{Moore normalisation} is a chain complex of groups
${\Moore} G_*=({\Moore} G_n, \pa_n)$, where $\Moore_0G_*=G_0$ while
\begin{align*}
{\Moore}_nG_*=\bigcap_{i=0}^{n-1} \Ker
d^n_i\quad\text{and}\quad \pa_n=d^n_n|_{\Moore_nG_*}, \quad n\geq 1.
\end{align*}
For $n\geq 1$, the \defn{$n$th homotopy group} of $G_*$ is defined as
\[
\pi_n(G_*)=\Ker\pa_n/\Im\pa_{n+1}.
\]

When an augmented simplicial group $(G_*,d^0_0, G)$ is given, then in dimension zero we calculate the homotopy group as
$\pi_0(G_*,d^0_0, G)=\Ker d^0_0/\Im\pa_1$. We say that the augmented simplicial group $(G_*,d^0_0, G)$ is
\defn{aspherical} if $\pi_n(G_*,d^0_0, G)=0$ for all $n\ge 0$ and $\Im d_0^0 = G$.

Given a crossed module $(H, G, \mu)$, the Moore complex of its nerve is trivial in dimensions $\geq 2$. In fact the Moore complex
of $\Nerve_*(H,G,\mu)$ is just the original crossed module up to isomorphism with $H$ in dimension $1$ and $G$ in dimension $0$.
Consequently, we have the isomorphisms
\begin{align}\label{equ1}
\pi_1 (\Nerve_*(H,G,\mu))\cong\Ker \mu\qquad\text{and}\qquad \pi_0(\Nerve_*(H,G,\mu))\cong G/\Im \mu.
\end{align}

Suppose we are given composable morphisms $(\rho,\nu)\colon (H,G,\mu)\to(H',G',\mu')$ and $(\rho',\nu')\colon (H',G',\mu')\to(H'',G'',\mu'')$ such that $\Ker\rho' =\rho$ and $\Ker\nu'=\nu$, while $\Coker \rho=\rho'$ and $\Coker \nu =\nu'$ in $\Gp$. This is precisely saying that $\Ker(\rho',\nu')=(\rho,\nu)$ and $\Coker(\rho,\nu)=(\rho',\nu')$ in $\XMod$, so that $(\rho,\nu)$ and $(\rho',\nu')$ form a \defn{short exact sequence} of crossed modules. In this case we write
\begin{align*}
\xymatrix{0 \ar[r] & (H,G,\mu) \ar[r]^-{(\rho,\nu)} & (H',G',\mu') \ar[r]^-{(\rho',\nu')} & (H'',G'',\mu'') \ar[r] & 0.}
\end{align*}
Recall that the nerve functor is exact, sending any short exact sequence of crossed modules to a short exact sequence of simplicial groups. 

If both $G$ and $H$ are abelian groups with trivial action of~$G$ on $H$, then $(H, G,\mu)$ is said to be an \defn{abelian} crossed module. A crossed module is abelian precisely when it is an an abelian group object in the category of crossed modules~\cite{Carrasco-Homology}. Given a crossed module $(H, G,\mu)$, a \defn{Beck module}~\cite{Beck} over $(H, G,\mu)$ is an abelian group object in the slice category $(\XMod\downarrow (H, G,\mu))$. Since the category of crossed modules satisfies the so-called \emph{Smith is Huq} condition~\cite{MFVdL,ACC}, such an internal abelian group object is the same thing as a split exact sequence
\begin{align*}
\xymatrix{0 \ar[r] & (M,P,\nu) \ar[r] & (H',G',\mu') \ar@<.5ex>[r] & (H,G,\mu) \ar@<.5ex>[l] \ar[r] & 0.}
\end{align*}
where $(M, P, \nu)$ is an abelian crossed module. As always in a semi-abelian category, the middle object $(H',G',\mu')$ in such a split exact sequence may be seen as a semidirect product $(M,P,\nu)\rtimes_{\xi} (H,G,\mu)$ where $\xi$ is an \emph{internal action}~\cite{Bourn-Janelidze:Semidirect,BJK} of $(H,G,\mu)$ on $(M,P,\nu)$. In the case of crossed modules, these internal actions agree with the concept of an action studied by Norrie~\cite{Norrie} and Forrester-Barker~\cite{FB}. Here we recall the presentation given in~\cite{CIKL}. An action of a crossed module $(H, G, \mu)$ on another crossed module $(M, P, \nu)$ is completely determined by an action of $G$ (and so $H$) on $M$ and $P$ together with a map $\xi \colon H\times P \to M$ such that the conditions 
\begin{align*}
\nu (^gm)&= \,^g \nu(m),
& ^{g}(^pm)&= \,^{^gp}m, \\
\nu \xi (h, p)&=\,^{\mu(h)}pp^{-1}, 
&\xi (h, \nu(m))&=\,^{\mu(h)}mm^{-1}, \\
^g\xi (h, p)&= \xi (^gh, \;^gp), 
&\xi (hh', p)&= \,^{\mu(h)}\xi (h', p)\xi (h, p), \\
\xi (h, pp')&= \xi (h, p)^p\xi (h, p')
\end{align*}
hold for all $g\in G$, $h$, $h'\in H$, $p$, $p'\in P$, $m\in M$. For instance, the trivial action of~$G$ on~$M$ and $P$
together with the trivial map $\xi \colon H\times P \to M$ provides an action of $(H, G, \mu)$ on $(M, P, \nu)$,
called the \defn{trivial action}.

Thus a morphism of $(H, G, \mu)$-modules $(\al, \be)\colon (M, P, \nu)\to (M', P', \nu')$ amounts to a pair of $G$-module homomorphisms
$\al \colon M\to M'$, $\be\colon P\to P'$ such that $\be\nu = \nu'\al$ and $\al \xi = \xi' (1_H\times \be)$. Note that given an action of $(H, G, \mu)$ on $(M, P, \nu)$ and a morphism of crossed modules ${(\bar{H}, \bar{G}, \bar{\mu})\to (H, G, \mu)}$, we naturally obtain an action of $(\bar{H}, \bar{G}, \bar{\mu})$ on $(M, P, \nu)$. Moreover, the conditions mentioned above imply that the action of $H$ on $\Ker \nu$ is trivial, while the
action of $H$ on $P$ is trivial modulo $\Im \nu$. Thus, we have naturally defined actions of $G/\Im \mu$ on both $\Ker \nu$ and
$P/\Im \nu$.

\section{ Homology of crossed modules with coefficients}\label{Homology}

Any action of a crossed module $(H, G, \mu)$ on another crossed module
$(M, P, \nu)$ induces a (pointwise) action of the simplicial group $\Nerve_*(H, G, \mu)$ on $\Nerve_*(M, P, \nu)$. Indeed, actions correspond to split extensions, and since the nerve functor preserves all limits, it does also preserve split extensions, and the corresponding actions.

Given a group $G$, let $B_*(G)=(B_*(G), \pa )$ denote the standard $G$-resolution over~$\Z$: we set $B_{-1}(G)=\Z$, and for $m\geq 0$,
\[
B_m(G)=\underbrace{\Z(G)\otimes \Z(G)\otimes \cdots \otimes \Z(G)}_{\text{$m+1$ factors}}
\]
where $\Z(G)$ denotes the group ring over $G$, and $\pa (g_0 \otimes g_1 \otimes \cdots \otimes g_m)$ is 
\begin{align*}
(-1)^{m}g_0 \otimes g_1\otimes\cdots \otimes g_{m-1}+\sum^{m-1}_{i=0}(-1)^{i}g_0 \otimes \cdots \otimes g_ig_{i+1}\otimes \cdots \otimes g_m.
\end{align*}
We know that $B_m(G)$ has a structure of a right $G$-module defined by
\begin{align*}
(g_0 \otimes g_1\otimes \cdots \otimes g_m)g =(g^{-1}g_0 \otimes g_1\otimes \cdots \otimes g_m).
\end{align*}

Suppose that $G_*=(G_n, d_n^i, s_n^i)$ is a simplicial group acting on an abelian simplicial group $A_*=(A_n, \pa_n^i, t_n^i)$, i.e., each
$A_n$ is an abelian group. Then for each $n\geq 1$ we have a map of complexes
$B_*(G_n)\otimes_{G_n} A_n \to B_*(G_{n-1})\otimes_{G_{n-1}} A_{n-1}$ defined by the alternating sum of face homomorphisms:
\begin{align*}
& b\otimes a \mapsto \sum^{n}_{i=0}(-1)^{i} d_n^i(b)\otimes \pa_n^i(a) ,
\end{align*}
for all $b\in B_m(G_n)$, $a\in A_n$. In this way we obtain the bicomplex of abelian groups
\[
\xymatrix{{} \ar[d] & {} \ar[d] & {} \ar[d]\\
B_2(G_{0})\otm_{G_{0}}A_{0} \ar[d] & B_2(G_{1})\otm_{G_{1}}A_{1} \ar[d] \ar[l] & B_2(G_{2})\otm_{G_{2}}A_{2} \ar[d] \ar[l] & \cdots \ar[l] \\
B_1(G_{0})\otm_{G_{0}}A_{0} \ar[d] & B_1(G_{1})\otm_{G_{1}}A_{1} \ar[d] \ar[l] & B_1(G_{2})\otm_{G_{2}}A_{2} \ar[d] \ar[l] & \cdots \ar[l] \\
B_0(G_{0})\otm_{G_{0}}A_{0} & B_0(G_{1})\otm_{G_{1}}A_{1} \ar[l] & B_0(G_{2})\otm_{G_{2}}A_{2} \ar[l] &\cdots \ar[l]}
\]
which we write $B(G_*, A_*)$.

\begin{definition} 
For $n\geq 0$, the \defn{$n$th homology group of $G_*$ with coefficients in $A_*$} is defined by the formula
\begin{align*}
& H_n(G_*, A_* )=H_n \big(\Tot B(G_*, A_*)\big).
\end{align*}
\end{definition}

This immediately gives a first quadrant spectral sequence
\begin{align*}
E^1_{pq}=H_q(G_p, A_p) \Rightarrow H_{p+q}(G_*, A_*),
\end{align*}
where $H_q(G_p, A_p)$ is the classical Eilenberg--Mac\,Lane homology.

\begin{definition} 
Let $(H, G, \mu)$ be a crossed module acting on an abelian crossed module $(C, A, \nu)$. Then,
\defn{the $n$th homology group of $(H, G, \mu)$ with coefficients in $(C, A, \nu)$} is defined by the formula
\begin{align*}
& H_n((H, G, \mu), (C, A, \nu))=H_n \big(\Tot B(\Nerve_*(H, G, \mu), \Nerve_*(C, A, \nu))\big) , \quad n\geq 0.
\end{align*}
\end{definition}

Suppose that a crossed module $(H, G, \mu)$ acts trivially on the inclusion crossed module $(0, \Z, 0)$. Then we write $H_n(H, G, \mu)$ instead of $H_n((H, G, \mu), (0, \Z, 0))$ and call it the \defn{$n$th homology group of $(H, G, \mu)$ with integral coefficients}. 

\begin{proposition}
For each $n\geq 0$, $H_{n}(H, G, \mu)$ is isomorphic to the $n$th homology of the crossed module $(H, G, \mu)$ with integral coefficients considered in~\cite{Baues1991} and~\cite{Ellis1992}.
\end{proposition}
\begin{proof}
For any (discrete) group, one can define its nerve by viewing it as a category with just one object. Let $\mathcal{N}_{**}(H, G, \mu)$ be the bisimplicial set obtained by forming the nerve of the group in each dimension of the simplicial group $\Nerve_*(H, G, \mu)$. Denote by $\Z\big(\mathcal{N}_{**}(H, G, \mu)\big)$ the bisimplicial abelian group obtained by forming the free abelian group pointwise in all degrees, i.e., $\Z\big(\mathcal{N}_{pq}(H, G, \mu)\big)$ is the free abelian group over $\mathcal{N}_{pq}(H, G, \mu)$ for each $p\geq 0$ and $q\geq 0$. Then for each $n\geq 0$, the $n$th homology group of the diagonal of $\Z\big(\mathcal{N}_{**}(H, G, \mu)\big)$ is isomorphic to the $n$th homology group of $(H, G, \mu)$ with integral coefficients defined by Baues and Ellis (see Proposition~2 in~\cite{Ellis1992}). Moreover, by taking the alternating sums of face homomorphisms of $\Z\big(\mathcal{N}_{**}(H, G, \mu)\big)$ in all vertical and horizontal directions, we regain the aforementioned bicomplex $B(\Nerve_*(H, G, \mu), \Nerve_*(0, \Z, 0))$. Whence the proposition.
\end{proof}

Let $G$ be a group acting on an abelian group $A$. Clearly then, the inclusion crossed module $(0, G, 0)$ acts on $(0, A, 0)$.

\begin{proposition} There is an isomorphism
\begin{align*}
H_n((0,G, 0),(0, A, 0))\cong H_n(G, A), \quad n\geq 0.
\end{align*}
\end{proposition}
\begin{proof} The bicomplex $B(\Nerve_*(0,G, 0), \Nerve_*(0, A, 0))$ has the following form:
\[
\xymatrix{{} \ar[d] & {} \ar[d] & {} \ar[d]\\
B_2(G)\otm_{G}A \ar[d] & B_2(G)\otm_{G}A \ar[d] \ar[l]_-{0} & B_2(G)\otm_{G}A \ar[d] \ar[l]_-{1} & \cdots \ar[l]_-{0} \\
B_1(G)\otm_{G}A \ar[d] & B_1(G)\otm_{G}A \ar[d] \ar[l]_-{0} & B_1(G)\otm_{G}A \ar[d] \ar[l]_-{1} & \cdots \ar[l]_-{0} \\
B_0(G)\otm_{G}A & B_0(G)\otm_{G}A \ar[l]_-{0} & B_0(G)\otm_{G}A \ar[l]_-{1} &\cdots \ar[l]_-{0}}
\]
Therefore $\Tot B(\Nerve_*(0,G, 0), \Nerve_*(0, A, 0))$ is quasi-isomorphic to the first column of
$B(\Nerve_*(0,G, 0), \Nerve_*(0, A, 0))$. Whence the proposition.
\end{proof}

In what follows, let $(C, A, \nu)$ denote an abelian crossed module. We will write the operations in $A$ and $C$ additively.

\begin{proposition}\label{prop1} Let $(H, G, \mu)$ be a crossed module acting on $(C, A, \nu)$. Then
\begin{align*}
H_0((H, G, \mu), (C, A, \nu))= A/(\nu (C)+\langle G, A \rangle),
\end{align*}
where $\langle G, A \rangle$ denotes the subgroup of $A$ generated by all $^ga-a$ for $g\in G$ and $a\in A$.
\end{proposition}
\begin{proof} We have a first quadrant spectral sequence
\begin{align*}
E^1_{pq}=H_q(\Nerve_p(H,G, \mu), \Nerve_p(C, A, \nu))\Rightarrow H_{p+q}((H,G, \mu), (C, A, \nu)).
\end{align*}
In particular
\begin{align*}
&E^1_{00}= H_0(G, A)=A/\langle G, A \rangle,\\
&E^1_{10}= H_0(H\rtimes G, C\rtimes A)=(C\rtimes A)/ \langle H\rtimes G, C\rtimes A \rangle ,
\end{align*}
and the differential $E^1_{10} \to E^1_{00}$ is induced by $(c, a)\mapsto \nu (c)$ for all $c\in C$ and $a\in A$. Thus
\begin{align*}
E^2_{00}=\Coker (E^1_{10} \to E^1_{00})=A/(\nu (C)+\langle G, A \rangle).
\end{align*}
Moreover $H_0((H, G, \mu), (C, A, \nu))= E^2_{00}$. Whence the proposition.
\end{proof}

\begin{proposition}\label{prop2} Let $(H, G, \mu)$ be a crossed module. Then
\begin{align*}
H_1(H, G, \mu)= G/(\mu (H)\cdot [G, G]).
\end{align*}
\end{proposition}
\begin{proof} We have a first quadrant spectral sequence
\begin{align*}
E^1_{pq}=H_q(\Nerve_p(H,G, \mu), \Z) \Rightarrow H_{p+q}(H,G, \mu).
\end{align*}
Thus $E^1_{0q}=\Z$ for $q\geq 0$ and $E^2_{0q}=0$ for $q\geq 1$. This implies that
\begin{align*}
H_{1}(H,G, \mu)=E^2_{01}.
\end{align*}
Since the first homology of a group with integral coefficients is isomorphic to the abelianisation of the same group,
we have:
\begin{align*}
E^1_{01}&= G/[G, G],\\
E^1_{11}&= (H\rtimes G)/[H\rtimes G, H\rtimes G],
\end{align*}
and the differential $E^1_{11} \to E^1_{01}$ is induced by $(h,g)\mapsto \mu(h)$. Thus
\begin{align*}
E^2_{01}= \Coker (E^1_{11} \to E^1_{01}) = G/(\mu(H)\cdot [G, G]),
\end{align*}
which finishes the proof.
\end{proof}

\begin{proposition} For any short exact sequence of $(H, G, \mu)$-modules
\begin{align*}
0\to (C, A, \nu)\to (C', A', \nu')\to (C'', A'', \nu'')\to 0,
\end{align*}
there is a long exact sequence of homology groups
\begin{align*}
&\cdots \to H_n((H, G, \mu), (C, A, \nu))\to H_n((H, G, \mu), (C', A', \nu'))\to \\
&H_n((H, G, \mu), (C'', A'', \nu''))\to H_{n-1}((H, G, \mu), (C, A, \nu))\to \cdots .
\end{align*}
\end{proposition}
\begin{proof} It is easy to see that the short exact sequence given in the proposition yields the following
short exact sequence of bicomplexes:
\begin{align*}
&0\to B(\Nerve_*(H, G, \mu), \Nerve_*(C, A, \nu))\to B(\Nerve_*(H, G, \mu), \Nerve_*(C', A', \nu'))\to \\
& \to B(\Nerve_*(H, G, \mu), \Nerve_*(C'', A'', \nu''))\to 0.
\end{align*}
The corresponding long exact homology sequence is the one claimed.
\end{proof}

For any simplicial group $G_*$ and any functor $T$ from the category of groups to the category of abelian groups,
let $T(G_*)$ denote the abelian simplicial group obtained by applying the functor $T$ dimension-wise to $G_*$.
The following lemma was offered to us by Nick Inassaridze.

\begin{lemma}\label{lemma1} Let $(G_*, d_0^0, G)$ be an augmented aspherical simplicial group. Then $(\Z(G_*), \Z(d_0^0), \Z(G))$
is also aspherical.
\end{lemma}
\begin{proof} 
We write $G_{-1}=G$. Since $(G_*, d_0^0, G)$ is aspherical, there is a sequence of maps (which need not be homomorphisms) $(h_n\colon G_n\to G_{n+1})_{n\geq -1}$ which forms a \emph{left contraction} for $(G_*, d_0^0, G)$; see~\cite{IH} for more details. Then the sequence $(\Z(h_n)\colon \Z(G_n)\to \Z(G_{n+1}))_{n\geq -1}$ will be a left contraction for $(\Z(G_*), \Z(d_0^0), \Z(G))$. Hence $(\Z(G_*), \Z(d_0^0), \Z(G))$
is aspherical.
\end{proof}

Let $G_*$ be a simplicial group acting on an abelian simplicial group $A_*$. For each $m\geq 0$, let $B_m(G_*, A_*)$ denote
the following abelian simplicial group:
\[
\xymatrix{\cdots \ar@<1.5ex>[r] \ar@<-1.5ex>[r] \ar@<.5ex>[r] \ar@<-.5ex>[r] & B_{m}(G_2)\otimes_{G_2} A_2 \ar@<1ex>[r] \ar[r] \ar@<-1ex>[r] & B_{m}(G_1)\otimes_{G_1} A_1 \ar@<.5ex>[r] \ar@<-.5ex>[r] & B_{m}(G_0)\otimes_{G_0} A_0.}
\]

\begin{lemma}\label{lemma2} Let $(G_*, d_0^0, G)$ be an augmented aspherical simplicial group. Then there is an isomorphism
\begin{align*}
\pi_n(B_m(G_*, A_*))\cong B_{m-1}(G)\otimes \pi_n (A_*).
\end{align*}
\end{lemma}
\begin{proof} For each $m\geq 0$ and $k\geq 0$ we have an isomorphism of abelian groups:
\begin{align*}
B_{m}(G_k)\otimes_{G_k} A_k \cong B_{m-1}(G_k)\otimes A_k.
\end{align*}
Thus, we need to calculate the $n$th homotopy group of the following simplicial abelian group:
\[
\xymatrix{\cdots \ar@<1.5ex>[r] \ar@<-1.5ex>[r] \ar@<.5ex>[r] \ar@<-.5ex>[r] & B_{m-1}(G_2)\otimes A_2 \ar@<1ex>[r] \ar[r] \ar@<-1ex>[r] & B_{m-1}(G_1)\otimes A_1 \ar@<.5ex>[r] \ar@<-.5ex>[r] & B_{m-1}(G_0)\otimes A_0.}
\]
This latter simplicial group is the diagonal of the bisimplicial group
\[
\xymatrix{& {} \ar@<1.5ex>[d] \ar@<-1.5ex>[d]\ar@<.5ex>[d] \ar@<-.5ex>[d] & {} \ar@<1.5ex>[d] \ar@<-1.5ex>[d]\ar@<.5ex>[d] \ar@<-.5ex>[d] & {} \ar@<1.5ex>[d] \ar@<-1.5ex>[d]\ar@<.5ex>[d] \ar@<-.5ex>[d] \\
\cdots \ar@<1.5ex>[r] \ar@<-1.5ex>[r]\ar@<.5ex>[r] \ar@<-.5ex>[r] & B_{m-1}(G_2)\otimes A_2 \ar@<1ex>[d] \ar[d] \ar@<-1ex>[d] \ar@<1ex>[r] \ar[r] \ar@<-1ex>[r] & B_{m-1}(G_1)\otimes A_2 \ar@<.5ex>[r] \ar@<-.5ex>[r] \ar@<1ex>[d] \ar[d] \ar@<-1ex>[d] & B_{m-1}(G_0)\otimes A_2 \ar@<1ex>[d] \ar[d] \ar@<-1ex>[d] \\
\cdots \ar@<1.5ex>[r] \ar@<-1.5ex>[r] \ar@<.5ex>[r] \ar@<-.5ex>[r] & B_{m-1}(G_2)\otimes A_1 \ar@<.5ex>[d] \ar@<-.5ex>[d] \ar@<1ex>[r] \ar[r] \ar@<-1ex>[r] & B_{m-1}(G_1)\otimes A_1 \ar@<.5ex>[r] \ar@<-.5ex>[r] \ar@<.5ex>[d] \ar@<-.5ex>[d] & B_{m-1}(G_0)\otimes A_1 \ar@<.5ex>[d] \ar@<-.5ex>[d] \\
\cdots \ar@<1.5ex>[r] \ar@<-1.5ex>[r] \ar@<.5ex>[r] \ar@<-.5ex>[r] & B_{m-1}(G_2)\otimes A_0 \ar@<1ex>[r] \ar[r] \ar@<-1ex>[r] & B_{m-1}(G_1)\otimes A_0 \ar@<.5ex>[r] \ar@<-.5ex>[r] & B_{m-1}(G_0)\otimes A_0.}
\]
Thus, by the Eilenberg--Zilber theorem~\cite{Weibel} and the K\"unneth formula~\cite{MacLane:Homology} we get:
\begin{align*}
\pi_n(B_m(G_*, A_*))= \underset{i+j=n}{\bigoplus}\pi_i (B_{m-1}(G_*))\otimes \pi_j (A_*).
\end{align*}
Moreover, by Lemma~\ref{lemma1} the augmented simplicial group
\[
(B_{m-1}(G_*), B_{m-1}(d_0^0), B_{m-1}(G))
\]
is aspherical. Whence the lemma.
\end{proof}

Given a group $G$ and a normal subgroup $N\normal G$, let $\sigma \colon N\to G$ be the natural inclusion. Suppose that $(N, G, \sigma)$ acts on an
abelian crossed module $(C, A, \nu)$. Recall that both $\Ker \nu$ and $A/\Im \nu$ have a $G/N$-module structure (see Section~\ref{sec1}). Thus, the homology groups $H_n(G/N, \Ker\nu)$ and $H_n(G/N, A/\Im\nu)$, for $n\geq 0$, discussed in the next proposition make sense.

\begin{proposition}\label{prop3} There is a long exact sequence
\begin{align*}
&\cdots \to H_{n-1}(G/N, \Ker \nu) \to H_n((N, G, \sigma), (C, A, \nu))\to H_n(G/N, A/\Im\nu)\\
&\to H_{n-2}(G/N, \Ker \nu) \to \cdots
\end{align*}
\end{proposition}
\begin{proof} We have a first quadrant spectral sequence
\begin{align*}
E^1_{pq}= H_q \big(B_p(\Nerve_*(N, G, \sigma),\Nerve_*(C, A, \nu))\big) \Rightarrow H_{p+q}((N, G, \sigma),(C, A, \nu)).
\end{align*}
By Lemma~\ref{lemma2} and~\eqref{equ1} we have:
\begin{align*}
E^1_{pq} =
\begin{cases}
0\qquad & \text{when $q\geq 2$,}\\
B_{p-1}(G/N)\otimes \Ker\nu & \text{when $q=1$,} \\
B_{p-1}(G/N)\otimes (A/\Im\nu) & \text{when $q=0$.}
\end{cases}
\end{align*}
Furthermore,
\begin{align*}
E^2_{p1}&=H_p(B_{*-1}(G/N)\otimes \Ker\nu)=H_p(B_{*}(G/N)\otimes_{G/N} \Ker\nu)=H_p(G/N, \Ker\mu), \\
E^2_{p0}&=H_p(B_{*-1}(G/N)\otimes (A/\Im\nu))=H_p(B_{*}(G/N)\otimes_{G/N} (A/\Im\nu)) \\
& =H_p(G/N, (A/\Im\nu)).
\end{align*}
This implies the proposition.
\end{proof}

It is well known that the category of crossed modules $\XMod$ may be considered as a \emph{variety of $\Omega$-groups} in the sense of Higgins~\cite{Higgins}. As explained in~\cite{Carrasco-Homology}, the forgetful functor $\U\colon \XMod\to \Set$ to the category of sets assigns, to a crossed module $(H, G, \mu)$, the cartesian product of the underlying sets of the groups $H$ and~$G$. Its left adjoint is constructed as follows. Given a group $G$, let $G+G$ be the coproduct (=~free product) of $G$ with itself, with injections $\iota_1$, $\iota_{2}\colon G\to G+G$, and let $i_{G}\colon {\overline{G}\to G+G}$ be the the normal closure of $\iota_{1}$, which may be obtained as the kernel of the retraction ${\lgroup 0\; 1_{G}\rgroup\colon G+G\to G}$ of $\iota_{2}$. This determines an inclusion crossed module
$(\overline{G}, G+G, i_{G})$. The functor $\F\colon \Set\to\XMod$ assigns, to any set~$X$, the inclusion crossed module
$(\overline{F_{X}}, F_{X}+F_{X}, i_{F_{X}})$, where $F_{X}$ is the free group over~$X$.

\begin{corollary}\label{cor} Let $(R, F, \eta)$ be a crossed module, free over $\Set$, acting on an abelian crossed module $(C, A, \mu)$. Then
\[
\begin{cases}
H_{n}((R, F, \eta), (C, A,\nu))=0, & n\geq 3, \\
H_{n}(R, F, \eta)=0, & n\geq 2.
\end{cases}
\]
\end{corollary}
\begin{proof} 
By the above, $(R, F, \eta)$ is an inclusion crossed module. Moreover, $F/R$ is a free group. Hence, the corollary results from Proposition~\ref{prop3}.
\end{proof}

\section{The Lyndon--Hochschild--Serre spectral sequence }

Let $G_*$ be a simplicial group acting on an abelian simplicial group $A_*$. For each $m\geq 0$ we have an abelian simplicial
group
\[
\xymatrix{\cdots \ar@<1.5ex>[r] \ar@<-1.5ex>[r] \ar@<.5ex>[r] \ar@<-.5ex>[r] & H_m(G_2,A_2) \ar@<1ex>[r] \ar[r] \ar@<-1ex>[r] & H_m(G_1,A_1) \ar@<.5ex>[r] \ar@<-.5ex>[r] & H_m(G_0, A_0),}
\]
denoted by $\HH_m(G_*, A_*)$.

\begin{remark}
$\HH_m(G_*, A_*)$ is not to be confused with $H_m(G_*, A_*)$. The latter is a group, while $\HH_m(G_*, A_*)$ is a simplicial group.
\end{remark}

Given a short exact sequence of simplicial groups
\begin{align*}
0\lra \Gm_* \lra \Pi_* \lra G_* \lra 0
\end{align*}
and an abelian simplicial group $A_*$ together with an action of $\Pi_*$ on $A_*$, we define an action of $G_*$ on
$\HH_m(\Gm_*, A_*)$ for each $m\geq 0$. First define an action of $\Pi_n$ on $B_m(\Pi_n)\otimes _{\Gm_n} A_n$ by the formula:
\begin{align*}
x(b\otm a)=bx^{-1}\otm xa, \quad x\in \Pi_n,\quad b\in B_m(\Pi_n),\quad a\in A_n.
\end{align*}
This action is $\Gm_n$-invariant. Hence, we have an induced action of $G_n$ on the group ${B_m(\Pi_n)\otimes _{\Gm_n} A_n}$.
Since the complex $B_*(\Pi_n)\otm_{\Gm_n} A_n$ calculates the homology groups
of $\Gm_n$ with coefficients in $A_n$, we have derived an action of $G_*$ on $\HH_m(\Gm_*, A_*)$, for each $m\geq 0$.

The following is an analogue of the Lyndon--Hochschild--Serre spectral sequence.

\begin{proposition}\label{prop4} 
There is a first quadrant spectral sequence
\begin{align*}
H_p(G_*, \HH_q(\Gm_*, A_*)) \Rightarrow H_{p+q}(\Pi_*, A_*).
\end{align*}
\end{proposition}
\begin{proof} 
Consider the following bicomplex:
\[
\xymatrix{{} \ar[d]_-{\partial} & {} \ar[d]_-{-\partial} & {} \ar[d]_-{\partial}\\
X_{01} \ar[d]_-{\partial} & X_{11} \ar[d]_-{-\partial} \ar[l] & X_{21} \ar[d]_-{\partial} \ar[l] & \cdots \ar[l] \\
X_{00} & X_{10} \ar[l] & X_{20} \ar[l] & \cdots, \ar[l]}
\]
where $(X_{p*}, \pa)=\Tot \big(B(G_p)\otm_{G_p} (B(\Pi_p)\otm_{\Gm_p}A_p)\big)$. Here the tensor product of two complexes
$B(G_p)$ and $B(\Pi_p)\otm_{\Gm_p}A_p$ is defined in the ordinary way. Then, the complex $X_{p*}$ is quasi-isomorphic to
$B(\Pi_p)\otm_{\Pi_p}A_p$ (see~\cite{MacLane:Homology}). Therefore, we have an isomorphism
\begin{align}\label{equ2}
H_n(\Tot (X_{**})) \cong H_n(\Pi_*, A_*), \quad n\geq 0 .
\end{align}
Furthermore, observe that
\begin{align*}
\Tot (X_{**})_n=\bigoplus_{p+q+l=n} B_p(G_l)\otm_{G_l} (B_q(\Pi_l)\otm_{\Gm_l}A_l) .
\end{align*}
We define the following filtration on $\Tot (X_{**})$:
\begin{align*}
F_m(\Tot (X_{**})_n)= \uns{p+l\leq m}{\uns{p+q+l=n}{\bigoplus}} B_p(G_l)\otm_{G_l} (B_q(\Pi_l)\otm_{\Gm_l}A_l) .
\end{align*}
The associated spectral sequence has the form
\begin{align*}
E^1_{mq}=\uns{p+l= m}{\bigoplus} B_p(G_l)\otm_{G_l}H_q(\Gm_l, A_l),
\end{align*}
and $E^2_{mq}$ is isomorphic to the $m$th homology group of the following bicomplex:
\[
\xymatrix{{} \ar[d] & {} \ar[d]\\
B_1(G_0)\otm_{G_0}H_q(\Gm_0, A_0) \ar[d] & B_1(G_1)\otm_{G_1}H_q(\Gm_1, A_1) \ar[d] \ar[l] & \cdots \ar[l] \\
B_0(G_0)\otm_{G_0}H_q(\Gm_0, A_0) & B_0(G_1)\otm_{G_1}H_q(\Gm_1, A_1) \ar[l] & \cdots \ar[l]}
\]
By definition this bicomplex calculates the homology groups of $G_*$ with coefficients in $\HH_q(\Gm_*, A_*)$.
Thus,
\begin{align*}
E^2_{mq}= H_m (G_*, \HH_q(\Gm_*, A_*)) \Rightarrow H_{m+q} (\Tot(X_{**})).
\end{align*}
The latter together with~\eqref{equ2} imply the proposition.
\end{proof}

\begin{theorem}\label{theo1} Let $(H,G,\mu)$ be a crossed module acting on an abelian crossed module $(C, A, \nu)$.
Then, for a short exact sequence of crossed modules
\begin{align*}
0\to (H',G',\mu')\to (H,G,\mu) \to(H'',G'',\mu'')\to 0,
\end{align*}
there is a first quadrant spectral sequence
\begin{align*}
&E^2_{pq}= H_p \big(\Nerve_*(H'',G'',\mu''), \HH_q(\Nerve_*(H',G',\mu'), \Nerve_*(C, A, \nu))\big)\\
&\Rightarrow H_{p+q} ((H,G,\mu), (C, A, \nu)).
\end{align*}
\end{theorem}
\begin{proof} Straightforward from the previous proposition.
\end{proof}

\begin{lemma}\label{lemma3} Let $(\al, \be)\colon (N, G, \sigma)\to (N', G', \sigma')$ be a morphism of inclusion
crossed modules such that the induced map $G/N\to G'/N'$ is an isomorphism. Suppose $\Nerve_*(N', G', \sigma')$ acts on an abelian
simplicial group $A_*$. Then there is an isomorphism
\begin{align*}
H_n(\Nerve_*(N, G, \sigma), A_*)\cong H_n(\Nerve_*(N', G', \sigma'), A_*), \quad n\geq 0,
\end{align*}
where $\Nerve_*(N, G, \sigma)$ acts on $A_*$ via $\Nerve_*(N, G, \sigma)\to \Nerve_*(N', G', \sigma')$.
\end{lemma}
\begin{proof} The morphism $(\al, \be)\colon (N, G, \sigma)\to (N', G', \sigma')$ induces a homomorphism of bicomplexes
$B(\Nerve_*(N, G, \sigma), A_*)\to B(\Nerve_*(N', G', \sigma), A_*)$. By Lemma~\ref{lemma2} we have
\begin{align*}
\pi_n(B_m(\Nerve_*(N, G, \sigma), A_*))&=B_{m-1}(G/N)\otimes \pi_n (A_*),\\
\pi_n(B_m(\Nerve_*(N', G', \sigma'), A_*))&=B_{m-1}(G'/N')\otimes \pi_n (A_*)
\end{align*}
for $m$, $n\geq 0$. Hence a spectral sequence argument implies that
\[
\Tot B(\Nerve_*(N, G, \sigma), A_*)\to \Tot B(\Nerve_*(N', G', \sigma), A_*)
\]
is a quasi-isomorphism.
\end{proof}

Recall that a morphism of crossed modules $(\al, \be)\colon (H, G, \mu)\to (H', G', \mu')$ is said to be a \defn{weak equivalence} when the induced morphisms $\al \colon \Ker \mu \to \Ker \mu'$ and $\be \colon G/\Im\mu \to G'/\Im\mu'$ are isomorphisms. As explained in~\cite{Garzon-Miranda}, these weak equivalences are actually part of a Quillen model category structure on $\XMod$; in particular, any homotopy equivalence of crossed modules is a weak equivalence. The following proposition shows that the homology groups of crossed modules are homotopical invariants. 

\begin{proposition}\label{prop6} 
Let $(\al, \be)\colon (H, G, \mu)\to (H', G', \mu')$ be a weak equivalence of crossed modules.
Suppose $(H', G', \mu')$ acts on an abelian crossed module $(C, A, \nu)$ and consider the induced action of
$(H, G, \mu)$ on $(C, A, \nu)$. Then there is an isomorphism
\begin{align*}
H_n((H, G, \mu), (C, A, \nu))\cong H_n((H', G', \mu'), (C, A, \nu)), \quad n\geq 0.
\end{align*}
\end{proposition}
\begin{proof} 
We have the following exact sequences of crossed modules:
\begin{align*}
& 0\to (\Ker\mu, 0, 0)\to (H, G, \mu)\to (H/\Ker\mu, G, \mu)\to 0 ,\\
& 0\to (\Ker\mu', 0, 0)\to (H', G', \mu')\to (H'/\Ker\mu', G', \mu')\to 0 .
\end{align*}
We define a bicomplex $X_{**}$ (respectively $X'_{**}$) as in Proposition~\ref{prop4} by setting 
$\Gm_*=\Nerve_*(\Ker\mu, 0, 0)$, $\Pi_*=\Nerve_*(H, G, \mu)$ and $G_*=\Nerve_*(H/\Ker\mu, G, \mu)$
(respectively $\Gm'_*=\Nerve_*(\Ker\mu', 0, 0)$, $\Pi'_*=\Nerve_*(H', G', \mu')$ and $G'_*=\Nerve_*(H'/\Ker\mu', G', \mu')$).
Then, the morphism $(\al, \be)\colon (H, G, \mu)\to (H', G', \mu')$ induces a homomorphism of bicomplexes $X_{**}\to X'_{**}$.
This homomorphism is compatible with filtration defined in Proposition~\ref{prop4}. We have the following spectral sequences:
\begin{align*}
E^2_{pq}&=H_p\big(\Nerve_*(H/\Ker\mu, G, \mu), \HH_q( \Nerve_*(\Ker\mu, 0, 0), \Nerve_*(C, A, \nu))\big) \\
&\Rightarrow H_{p+q}(\Tot (X_{**}))= H_{p+q}((H, G, \mu), (C, A, \nu)),\\
E'^2_{pq}&=H_p\big(\Nerve_*(H'/\Ker\mu', G', \mu'), \HH_q( \Nerve_*(\Ker\mu', 0, 0), \Nerve_*(C, A, \nu))\big)\\
& \Rightarrow H_{p+q}(\Tot (X'_{**}))= H_{p+q}((H', G', \mu'), (C, A, \nu)).
\end{align*}
Since both $(H/\Ker\mu, G, \mu)$ and $(H'/\Ker\mu', G', \mu')$ are inclusion crossed modules, Lemma~\ref{lemma3} yields an isomorphism
$E^2_{pq}\cong E'^2_{pq}$ for each $p\geq 0$, $q\geq 0$. This implies that $\Tot (X_{**})$ and $\Tot (X'_{**})$ are quasi-isomorphic.
\end{proof}
Another consequence of the Lyndon--Hochschild--Serre spectral sequence is a five-term exact sequence relating homologies of
crossed modules in lower dimensions. Recall that for each crossed module $(H, G, \mu)$ its abelianisation is given by
\begin{align*}
(H, G, \mu)_{ab}=(H/[G, H], G/[G,G], \mu),
\end{align*}
where $[G, H]\normal H$ is the normal subgroup of $H$ generated by all $^ghh^{-1}$ for $g\in G$, $h\in H$ (we refer the reader to~\cite{Carrasco-Homology} to see that this is correct from the categorical point of view). It is easy to see that if $(H,G,\mu)$ acts trivially on $(0, \Z, 0)$
then
\begin{align}\label{equ3}
\HH_1(\Nerve_*(H,G,\mu), \Nerve_*(0, \Z, 0))= \Nerve_*(H,G,\mu)_{ab}.
\end{align}

\begin{proposition}\label{prop7} Given a short exact sequence of crossed modules
\begin{align*}
0\to (H',G',\mu')\to (H,G,\mu) \to(H'',G'',\mu'')\to 0,
\end{align*}
we have the following exact sequence:
\begin{align*}
&H_2(G, H, \mu)\to H_2(H'', G'', \mu'') \to G'/(\mu'(H')[G, G']) \\
&\to G/(\mu (H)[G, G]) \to G''/(\mu (H'')[G'', G'']) \to 0
\end{align*}
Moreover, there is a homomorphism $H_3(H, G, \mu)\to H_1((H'', G'', \mu''), (H', G', \mu')_{ab})$ and an epimorphism
from
\[
\Ker \left(H_2(G, H, \mu)\to H_2(G'', H'', \mu'')\right)
\]
to
\[
\Coker \left(H_3(H, G, \mu)\to H_1((H'', G'', \mu''), (H', G', \mu')_{ab})\right).
\]
\end{proposition}
\begin{proof} 
Suppose that $(H,G,\mu)$ acts trivially on $(0, \Z, 0)$. Then the Lyndon--Hoch\-schild--Serre spectral sequence
assumes the following form:
\[
\resizebox{\textwidth}{!}
{$E^2_{pq}= H_p \big(\Nerve_*(H'',G'',\mu''), \HH_q(\Nerve_*(H',G',\mu'), \Nerve_*(0,\Z, 0))\big) \Rightarrow
H_{p+q} (H,G,\mu)$.}
\]
Since this is a first quadrant spectral sequence, we have the following exact sequence:
\begin{align*}
H_2(H, G, \mu) \to E^2_{20} \to E^2_{01} \to H_1(H, G, \mu)\to E^2_{10} \to 0.
\end{align*}
Clearly $E^2_{20}=H_2(H'',G'',\mu'')$ and $E^2_{10}=H_1(H'',G'',\mu'')$. By Proposition~\ref{prop1}, Proposition~\ref{prop2}
and~\eqref{equ3} we have:
\begin{align*}
&E^2_{01}= H_0((H'',G'',\mu''), \Nerve_*(H',G',\mu')_{ab})=G'/(\mu'(H')[G, G']), \\
&H_1(H, G, \mu)=G/(\mu(H)[G, G]), \\
&E^2_{10}=H_1(H'',G'',\mu'')= G''/(\mu(H'')[G'', G'']).
\end{align*}
This implies the first part of the proposition.

Furthermore, since $E^2_{pq}$ is a first quadrant spectral sequence, we have an epimorphism:
\begin{align*}
\Ker\left( H_2(H, G, \mu) \to E^2_{20}\right) \to E^\infty_{11},
\end{align*}
and
\begin{align*}
E^\infty_{11} = \Coker\left( E^2_{30} \to E^2_{11}\right).
\end{align*}
Substituting $E^2_{30}$ and $E^2_{11}$ for the suitable homology groups, we obtain the desired result.
\end{proof}

Given a crossed module $(H, G, \mu)$, a \defn{free presentation} of $(H, G, \mu)$ is defined to be a short exact
sequence of crossed modules
\begin{align}\label{equ4}
0\to (R', F', \eta') \to (R, F, \eta) \to (H, G, \mu) \to 0,
\end{align}
where $(R, F, \eta)$ is a free crossed module (with respect to $\Set$).

\begin{corollary}[Hopf Formula] \label{cor1}
For any free presentation~\eqref{equ4}, there is an isomorphism:
\begin{align*}
H_2(H, G, \mu) \cong \frac{(\eta(R)\cdot [F, F])\cap F'}{\eta '(R')\cdot [F, F']}.
\end{align*}
\end{corollary}
\begin{proof} Straightforward from Proposition~\ref{prop7} and Corollary~\ref{cor}.
\end{proof}

\begin{corollary} For any crossed module $(H, G, \mu)$, there exists an epimorphism $H_2(H, G, \mu)\twoheadrightarrow H_2 (G/\mu(H))$.
\end{corollary}
\begin{proof} We have the following short exact sequence of crossed modules:
\begin{align*}
0\to (\Ker\mu, 0, 0) \to (H, G, \mu)\to (H/\Ker\mu, G, \mu)\to 0.
\end{align*}
This implies the existence of an epimorphism $H_2(H, G, \mu)\twoheadrightarrow H_2(H/\Ker\mu, G, \mu)$. Moreover, since
$(H/\Ker\mu, G, \mu)$ is an inclusion crossed module, $H_2(H/\Ker\mu, G, \mu)$ is isomorphic to $H_2(G/\mu(H))$.
\end{proof}

In Section~\ref{Higher} we shall consider higher-order versions of these two results.

\begin{proposition}\label{prop8} For any free presentation \eqref{equ4}, there is an isomorphism
\begin{align*}
H_3(H, G, \mu) \cong H_1((H, G, \mu), (R', F', \eta')_{ab}).
\end{align*}
\end{proposition}
\begin{proof} By Corollary~\ref{cor} we have $H_n(R, F, \eta)=0$, $n\geq 2$.
Therefore, from the Lyndon--Hochschild--Serre spectral sequence,
\begin{align*}
E^2_{pq}= H_p \big(\Nerve_*(H,G,\mu), \HH_q(\Nerve_*(R',F',\eta'), \Nerve_*(0,\Z, 0))\big) \Rightarrow
H_{p+q} (R,F,\eta),
\end{align*}
we derive that $E^2_{30}\to E^2_{11}$ is an epimorphism and $\Ker(E^2_{30}\to E^2_{11})$ must be isomorphic to $E^3_{02}$.
Since $E^2_{30}=H_3(H, G, \mu)$ and $E^2_{11}=H_1((H, G, \mu), (R', F', \eta')_{ab})$, it suffices to show that $E^2_{02}=0$.
By the definition $\HH_q(\Nerve_*(R',F',\eta'), \Nerve_*(0,\Z, 0))$ has the following form:
\[
\xymatrix{\cdots \ar@<1.5ex>[r] \ar@<-1.5ex>[r] \ar@<.5ex>[r] \ar@<-.5ex>[r] & H_m(R'\rtimes R'\rtimes F',\Z) \ar@<1ex>[r] \ar[r] \ar@<-1ex>[r] & H_m(R'\rtimes F', \Z) \ar@<.5ex>[r] \ar@<-.5ex>[r] & H_q(F', \Z).}
\]
Since $F'$ is a free group, $H_q(F', \Z)=0$, $q\geq 2$. Consequently
\begin{align*}
E^2_{0q}= H_0 \big(\Nerve_*(H,G,\mu), \HH_q(\Nerve_*(R',F',\eta'), \Nerve_*(0,\Z, 0))\big) =0
\end{align*}
for all $q\geq 2$.
\end{proof}

\section{Homology via non-abelian left derived functors}\label{Section Left Derived Functors}

We come back to the description of the forgetful/free adjunction between $\XMod$ and $\Set$ recalled just above Corollary~\ref{cor}. The pair of adjoint functors $(\F, \U)$ induces a comonad $\FF=(\FF, \de, \epsilon)$ on $\XMod$ in the usual way:
$\FF=\F\U\colon {\XMod\to \XMod}$, $\epsilon \colon \FF\to 1_{\XMod}$ is the counit and
$\de = \F \eta\U\colon \FF\to \FF^2$ where $\eta\colon 1_{\Set}\to \U\F$ is the unit of the adjunction.
Given any crossed module $(H, G, \mu)$, there is an augmented simplicial object $\FF_*(H, G, \mu)\to (H, G, \mu)$ in the category
$\XMod$ where
\begin{align*}
&\FF_n(H, G, \mu)= \FF^{n+1}(H, G, \mu), \\
&d_i^n=\FF^i(\epsilon (\FF^{n-i}(H, G, \mu))), \quad s_i^n=\FF^i(\de (\FF^{n-i}(H, G, \mu))), \quad 0\leq i\leq n.
\end{align*}
This is called the \defn{$\FF$-cotriple resolution of $(H, G, \mu)$}.

Let $\Ab$ denote the category of abelian groups and let $\T\colon \XMod\to \Ab$ be a functor. As in~\cite{Barr-Beck}, the \defn{left derived functors} of $T$ with respect to the comonad $P$ are given, for any crossed module $(H, G, \mu)$, by
\begin{align*}
&\LL_n\T(H, G, \mu)=\pi_n(\T(\FF_*(H, G, \mu))),
\end{align*}
where $\T(\FF_*(H, G, \mu))$ is the simplicial abelian group obtained by applying the functor $\T$ dimension-wise to the $\FF$-cotriple resolution of $(H, G, \mu)$. Note that the homotopy groups $\pi_n(\T(\FF_*(H, G, \mu)))$ agree with the homology groups of the simplicial abelian group $\T(\FF_*(H, G, \mu))$. The functors $\LL_n\T$, for $n\geq 0$, may also be interpreted as \emph{non-abelian left derived functors}, in the sense of~\cite{IH}, of $T$ relative to the projective class $\PP$ determined by the comonad $\FF$.

Let $\A\colon \XMod\to \Ab$ be the functor given by
\begin{align}\label{Functor A}
&(H, G, \mu) \mapsto \A(H, G, \mu)=G/(\mu(H)\cdot [G, G]).
\end{align}
Our goal is to show that $\LL_n\A(H, G, \mu)$ is isomorphic to $H_{n+1}(H, G, \mu)$ for each $n\geq 0$.
Since the functor $\A\colon \XMod\to \Ab$ is right exact, by Proposition~\ref{prop2} we have:
\begin{align*}
&\LL_0\A(H, G, \mu) = \A(H, G, \mu)=G/(\mu(H)[G, G])=H_1(H, G, \mu).
\end{align*}

\begin{lemma}\label{lemma4} For each $n\geq 0$, the following augmented simplicial group is aspherical:
\[
\xymatrix{\cdots \ar@<1ex>[r] \ar[r] \ar@<-1ex>[r] & \Nerve_n(\FF_1(H, G, \mu)) \ar@<.5ex>[r] \ar@<-.5ex>[r] & \Nerve_n(\FF_0(H, G, \mu)) \ar[r] & \Nerve_n(H, G, \mu).}
\]
Here $\Nerve_*(\FF_i(H, G, \mu))$ is the nerve of the crossed module $\FF_i(H, G, \mu)$.
\end{lemma}
\begin{proof} By the definition each $\FF_i(H, G, \mu)$, $i\geq 0$, is an inclusion crossed module. Thus we may write 
$\FF_i(H, G, \mu)=(R_i, F_i, \eta)$, $i\geq 0$, where the $R_{i}$ and $F_{i}$ are free. Then, both
\[
\xymatrix{\cdots \ar@<1.5ex>[r] \ar@<-1.5ex>[r] \ar@<.5ex>[r] \ar@<-.5ex>[r] & F_{2} \ar@<1ex>[r] \ar[r] \ar@<-1ex>[r] & F_{1} \ar@<.5ex>[r] \ar@<-.5ex>[r] & F_{0} \ar[r] & G}
\]
and
\[
\xymatrix{\cdots \ar@<1.5ex>[r] \ar@<-1.5ex>[r] \ar@<.5ex>[r] \ar@<-.5ex>[r] & R_{2} \ar@<1ex>[r] \ar[r] \ar@<-1ex>[r] & R_{1} \ar@<.5ex>[r] \ar@<-.5ex>[r] & R_{0} \ar[r] & H}
\]
are aspherical augmented simplicial groups (see~\cite{Carrasco-Homology}). This implies the lemma.
\end{proof}

\begin{theorem}\label{theo2} There is an isomorphism
\begin{align*}
\LL_n\A(H, G, \mu) \cong H_{n+1}(H, G, \mu), \quad n\geq 0.
\end{align*}
\end{theorem}
\begin{proof}
Consider the bicomplex
\[
\xymatrix{{} \ar[d]_-{\partial} & {} \ar[d]_-{-\partial} & {} \ar[d]_-{\partial}\\
X_{01} \ar[d]_-{\partial} & X_{11} \ar[d]_-{-\partial} \ar[l] & X_{21} \ar[d]_-{\partial} \ar[l] & \cdots \ar[l] \\
X_{00} & X_{10} \ar[l] & X_{20} \ar[l] & \cdots, \ar[l]}
\]
where $(X_{p*}, \pa)=\Tot B\big(\Nerve_*(\FF_p(H, G, \mu)), \Nerve_*(0, \Z, 0)\big)$. Lemma~\ref{lemma2} together with Lemma~\ref{lemma4} imply that $\Tot X_{**}$ and $\Tot B\big(\Nerve_*(H, G, \mu), \Nerve_*(0, \Z, 0)\big)$ are quasi-isomorphic.
As a consequence,
\begin{align}\label{equ5}
H_n(\Tot X_{**})=H_n(H, G, \mu), \quad n\geq 0.
\end{align}
On the other hand we have the following spectral sequence:
\begin{align*}
E^1_{pq}=H_q(X_{p*})=H_q(\FF_p(H, G, \mu)) \Rightarrow H_{p+q}(\Tot X_{**}).
\end{align*}
Since $\FF_p(H, G, \mu)$ is a free crossed module for all $p\geq 0$, by Proposition~\ref{prop1}, Proposition~\ref{prop2} and
Corollary~\ref{cor} we have:
\begin{align*}
E^1_{pq} =
\begin{cases}
\Z & \text{when $q=0$,}\\
\A(\FF_p(H, G, \mu)) & \text{when $q=1$,} \\
0 & \text{when $q\geq 2$.}
\end{cases}
\end{align*}
This implies that
\begin{align*}
E^2_{pq} =
\begin{cases}
\Z & \text{when $p=0$ and $q=0$,}\\
0 & \text{when $p>0$ and $q=0$,}\\
\pi_p(\A(\FF_*(H, G, \mu))) & \text{when $q=1$,} \\
0 & \text{when $q\geq 2$.}
\end{cases}
\end{align*}
Since $E^\infty_{pq}=E^2_{pq}$, we derive an isomorphism:
\begin{align}\label{equ6}
H_{n+1}(\Tot X_{**})\cong \pi_n(\A(\FF_*(H, G, \mu)))\cong \LL_n\A(H, G, \mu), \quad n\geq 0.
\end{align}
Thus,~\eqref{equ5} and~\eqref{equ6} complete the proof.
\end{proof}

\section{Higher Hopf formulae}\label{Higher}

Theorem~\ref{theo2} may be used to obtain Hopf formulae for $H_{n+1}(H, G, \mu)$. Indeed, the functor $\A$ is of a type considered in~\cite{Everaert-Gran-TT}, so that the general theory developed there applies.
 
Given a crossed module $(H, G, \mu)$, let $P_{*}(H, G, \mu)=(R_{*},F_{*},\eta)$ be its $P$-cotriple resolution. Let us fix some $n\geq 1$, and write $\langle n\rangle=\{0,\dots,n-1\}$. Following~\cite{EGVdL,EGoeVdL}, the $n$-truncation $(R_{i},F_{i},\eta)_{i< n}$ of this resolution may be considered as an \defn{$n$-fold free presentation} of $(H, G, \mu)$. We shall write $(R,F,\eta)$ for $(R_{n-1},F_{n-1},\eta)$, and denote the kernel of
\[
d^{n}_{i}\colon (R_{n-1},F_{n-1},\eta)\to (R_{n-2},F_{n-2},\eta)
\]
for $i\in \langle n\rangle$ by $(R'_{i},F'_{i},\eta'_{i})\normal (R,F,\eta)$. Then we find the following:

\begin{theorem}[Hopf Formula]\label{Theorem Hopf}
There is an isomorphism
\begin{align*}
H_{n+1}(H, G, \mu) \cong \frac{(\eta(R)\cdot [F, F])\cap \bigcap_{i\in n}F'_{i}}
{\eta (\bigcap_{i\in \langle n\rangle}R'_{i})\cdot \prod_{I\subseteq \langle n\rangle}[\bigcap_{i\in I}F'_{i}, \bigcap_{i\not\in I}F'_{i}]}.
\end{align*}
\end{theorem}
\begin{proof}
This is an application of Corollary~6.12 in~\cite{Everaert-Gran-TT}. In fact the situation we are considering is an instance of an example given on page~3674 of that paper, in the subsection entitled \emph{Internal groupoids with coefficients in abelian objects}. Via the equivalence between internal categories and crossed modules, it is the special case of that example for $\mathcal{A}=\Gp$. 

The functor $\A\colon \XMod\to \Ab\colon (H, G, \mu) \mapsto G/(\mu(H)\cdot [G, G])$ may thus be seen as the reflector of $\XMod$ to an intersection of two of its subvarieties: $\Ab(\XMod)$ on the one hand, and $\Gp$ (discrete crossed modules) on the other. The reflector to~$\Gp$ is \emph{protoadditive}, being the connected components functor viewed as ${(H,G,\mu)\mapsto G/\mu(H)}$. Hence the requirements of~\cite[Corollary~6.12]{Everaert-Gran-TT} are satisfied.
\end{proof}

Note how this particularises to Corollary~\ref{cor1} when $n=1$. Similarly, the first part of Proposition~\ref{prop7} follows from a result in \cite{EverVdL2} via the interpretation in terms of cotriple homology. In fact, it is the tail of a long exact sequence: see~\cite{TomasThesis, GVdL2}.

\providecommand{\noopsort}[1]{}

\end{document}